\documentclass[11pt]{amsart}

\usepackage{amssymb}
\usepackage{amsthm}
\usepackage{amsmath, amscd}
\usepackage{amsfonts}
\usepackage{easy}
\usepackage{graphicx}
\usepackage[thinlines]{easymat}
\usepackage{epic,eepic}
\usepackage{bm}
\usepackage{caption}
\usepackage{yfonts}
\usepackage{epsfig}
\usepackage{rotating}
\usepackage{blindtext}
\usepackage{bm}
\usepackage{enumerate}

\newcommand{\la}{\lambda}

\newcommand{\al}{\alpha}

\newcommand{\e}{\varepsilon}

\newcommand{\Q}{\mathbb{Q}}
\newcommand{\ov}{\overline}

\newcommand{\gl}{\mathfrak{g}}

\newcommand{\s}{\mathfrak{s}}

\renewcommand{\l}{\mathfrak{l}}
\renewcommand{\t}{\mathfrak{t}}

\newcommand{\F}{\mathbb F}
\newcommand{\Z}{\mathbb Z}

\renewcommand{\i}{\iota}
\renewcommand{\la}{\langle}
\newcommand{\ra}{\rangle}

\DeclareMathOperator{\codim}{codim}
\DeclareMathOperator{\support}{supp}
\DeclareMathOperator{\rank}{rank}

\DeclareMathOperator{\sgn}{sgn}
\DeclareMathOperator{\uni}{uni}
\DeclareMathOperator{\nil}{nil}
\DeclareMathOperator{\Irr}{Irr}
\DeclareMathOperator{\Hom}{Hom}
\DeclareMathOperator{\End}{End}
\DeclareMathOperator{\Ad}{Ad}
\DeclareMathOperator{\Lie}{Lie}
\DeclareMathOperator{\cl}{cl}
\DeclareMathOperator{\Tr}{Tr}

\DeclareMathOperator{\Sn}{S}
\DeclareMathOperator{\GL}{GL}
\DeclareMathOperator{\SL}{SL}
\DeclareMathOperator{\GU}{GU}

\newtheorem{theorem}{Theorem}[section]

\newtheorem{proposition}[theorem]{Proposition}
\newtheorem{lemma}[theorem]{Lemma}

\newtheorem{rmk}[theorem]{Remark}

\begin{document}

\title[On the endomorphism algebra of GGGRs]{On the endomorphism algebra of generalised Gelfand-Graev representations}

\author{MATTHEW C. CLARKE}

\address{Trinity College,
Cambridge, CB\textup{2 1}TQ, UK}

\email{\texttt{m.clarke@dpmms.cam.ac.uk}}

\thispagestyle{empty}

\begin{abstract}

Let $G$ be a connected reductive algebraic group defined over the finite field $\F_q$, where $q$ is a power of a good prime for $G$, and let $F$ denote the corresponding Frobenius endomorphism, so that $G^F$ is a finite reductive group. Let $u \in G^F$ be a unipotent element and let $\Gamma_u$ be the associated generalised Gelfand-Graev representation of $G^F$. Under the assumption that $G$ has a connected centre, we show that the dimension of the endomorphism algebra of  $\Gamma_u$ is a polynomial in $q$, with degree given by $\dim C_G(u)$. When the centre of $G$ is disconnected, it is impossible, in general, to parametrise the (isomorphism classes of) generalised Gelfand-Graev representations independently of $q$, unless one adopts a convention of considering separately various congruence classes of $q$. Subject to such a convention we extend our result.

\end{abstract}

\maketitle




\section{Introduction}  \label{intro}


In \cite{Kawa} N. Kawanaka has associated a generalised Gelfand-Graev representation to each unipotent class of a finite reductive group $G^F$. These representations have deep connections with the geometry of the unipotent classes of $G$. In \cite{Kawa} Kawanaka has also obtained a formula for their characters in terms of Green polynomials in the case of general linear and unitary groups. Inspired by this, G. Lusztig (\cite{Lus}) has obtained a similar formula, valid for an arbitrary finite reductive group (with $p$ sufficiently large), expressed in terms of intersection cohomology complexes of closures of unipotent classes with coefficients in various local systems. By considering the inner product of generalised Gelfand-Graev characters (with themselves) using these formulas we thus obtain expressions for the dimension of the associated generalised Gelfand-Graev modules.

Naturally, the prime power $q$ features heavily in these formulas and in such situations it is useful to think of $q$ as a variable. However care is needed in order to formulate precise and meaningful statements. This notion is central in the theory of finite reductive groups as it allows generic behaviour to be observed, i.e. behaviour which is independent of the associated finite field. Recently S. Goodwin and G. R\"{o}hrle have written two papers (\cite{GoRo2}, \cite{GoRo}) in which the notion of polynomials in $q$ is used extensively, and the set-up in the present paper is inspired by the set-up in those papers.

Under a certain assumption on the root datum of $G$ (ensuring that the centre of $G$, together with all groups with the same root datum but with a possibly different associated prime power $q$, are connected), we shall explain what we mean by the statement ``the dimension of a generalised Gelfand-Graev module is a polynomial in $q$'', before proving it. We will also show that the degree of this polynomial is given by the dimension of the centraliser (in $G$) of a unipotent element in $G^F$ from the class associated with that module. (For groups with a component of Type $E_8$ we may need to dichotomise our set-up, depending on the value of $q$ modulo $3$.) When $G$ has a disconnected centre we cannot even parametrise the generalised Gelfand-Graev characters independently of $q$ in general, and so we cannot hope for such a clean statement here. However, similar behaviour is exhibited in the disconnected centre case, and we have found that a suitable way to capture this is to consider $q$ as a variable, but whilst only allowing certain congruence classes of $q$. Subject to this restriction we extend the aforementioned results to this situation too.

In \cite{St} R. Steinberg has shown that (ordinary) Gelfand-Graev representations are multiplicity-free (although this was proved previously for split groups by T. Yokonuma in \cite{Yok1}, \cite{Yok2}), and that they contain $|Z(G)^F|q^l$ irreducible constituents. It thus follows that this number is the dimension of their endomorphism algebras. When the centre of $G$ is connected this number may be viewed as a polynomial in $q$ of degree $\rank G$, the latter agreeing with the dimension of the centraliser of a regular unipotent element. One could view the results of this paper as a generalisation of this fact to generalised Gelfand-Graev representations.

This paper is organised as follows. In the remainder of this section we lay down the rigorous foundation necessary to formulate precise statements involving polynomials in $q$. In Section \ref{Kawanaka} we prove our main result in the special case of general linear and unitary groups, using Kawanaka's character formula. In Section \ref{LusSec} we extend this to a much more general setting, using Lusztig's character formula. Although Section \ref{Kawanaka} is essentially a special case of Section \ref{LusSec}, it is considerably easier to understand, conceptually, the ingredients of the former, and may serve to illuminate the latter.

We let $G_{\uni}$ denote the unipotent variety of $G$, and for $\gl=\Lie(G)$, we let $\gl_{\nil}$ denote the nilpotent variety of $\gl$. Throughout, for $u \in G_{\uni}^F$, we denote by $\Gamma_u$ the corresponding generalised Gelfand-Graev representation, and the lower case $\gamma_u$ for its character. All representations will be over $\bar{\Q}_l$.

The author is grateful to M. Geck for suggesting the problem which this paper addresses and for many useful discussions. (The problem was suggested to him previously by A. Premet.) He also thanks Aberdeen University for its hospitality during his visit there in 2009, which was supported by the EPSRC Network Grant EP/F029381/1.

\subsection{Polynomials in $q$}

Recall that a connected reductive algebraic group is uniquely determined by its root datum $(X,Y,\Phi,\check{\Phi})$, and that a related finite reductive group is uniquely determined by the following.
\begin{enumerate}
	\item A root datum $(X,Y,\Phi,\check{\Phi})$.
	\item A prime power $q$.
	\item An automorphism $F_0$ of the Dynkin diagram associated with \\$(X,Y,\Phi,\check{\Phi})$.
\end{enumerate}


Now suppose we have a quantity attached to a fixed choice of data 1 and 3 above, which is a function of various prime powers $q$. If there exists a polynomial $f \in \Q[x]$ such that the quantity is given by $f(q)$ for those $q$ under consideration, then we say that this quantity is a {\em polynomial in $q$}. In the same spirit, we may also talk about quantities which are {\em independent of $q$}. When we write $G$ we will sometimes mean a fixed group with an associated fixed prime power $q$ (sometimes the notation $G(q)$ is used), but sometimes, by abuse, we will talk about properties of $G$ which are independent of $q$, in which case we are, strictly speaking, referring to properties of the root datum, together with $F_0$.


For the statement ``$\dim \End_{\bar{\Q}_lG^F} \Gamma_u$ is a polynomial in $q$'' to be meaningful, we first need to establish that the number of distinct generalised Gelfand-Graev characters is fixed as we vary $q$, and that they can be parametrised independently of $q$. But since the generalised Gelfand-Graev characters are parametrised by the unipotent classes of $G^F$ we shall focus on the latter. Clearly we will need to fix, once and for all, a root datum $(X,Y,\Phi,\check{\Phi})$, and automorphism $F_0$.

Before we begin the discussion of unipotent classes we fix some data which will play a part both in the current situation and later on when we will wish to compare $G^F$-classes of certain rational subgroups of $G$, across different values of $q$. We fix a maximally split maximal torus $T \le G$ for each prime power, and a simple system $\pi \subset \Phi$ such that $F_0(\pi) = \pi$. Then this uniquely determines a rational Borel subgroup $B \supset T$. 



In order to parametrise geometric unipotent classes and $F$-stable geometric unipotent classes independently of $q$ we apply an idea of N. Spaltenstein of using a group $G'$ with the desired root datum as $G$ but over a field of characteristic zero, as a reference point. We let $T'$ be a maximal torus of $G'$ and $B'$ a Borel subgroup of $G'$ containing $T'$. For this discussion we restrict our attention to $q$ which are powers of a good prime for our fixed root datum. (Recall that this restriction is also necessary to define generalised Gelfand-Graev representations.) Let $X_G$ and $X_{G'}$ denote the set of geometric unipotent classes of $G$ and $G'$ respectively. Then Spaltenstein has shown (\cite{Spa}, Th\'{e}or\`{e}me III.5.2) that there exists a map $\pi_G : X_{G'} \rightarrow X_{G}$, which is characterised by the following three properties:

\begin{itemize}
	\item It is an isomorphism of posets. 
	
	\item It preserves the dimensions of classes.
	
	\item It satisfies certain compatibility relations between parabolic subgroups in $G$ and $G'$ containing $B$ and $B'$ respectively.
\end{itemize}

Since this map is uniquely defined we may use it as a means of parametrising the geometric unipotent classes of $G$ independently of $q$. Now the Frobenius endomorphism $F$ on $G$ may be written as $F=F_q \circ F_0 = F_0 \circ F_q$ where $F_q$ is determined by the multiplication by $q$ map on the character group of $T$, whilst $F_0$ is an automorphism of $G$ of finite order, determined uniquely by the automorphism on $\Phi$ that we fixed previously. Then $F_0$ determines also a map $F_0'$ on $G'$ in the same manner. Clearly the maps $F$ and $F_0'$ induce permutations on $X_G$ and $X_{G'}$ respectively and, moreover, the following diagram commutes:

\[\begin{CD}
X_{G'} @>\pi_G >> X_G\\
@VV{F_0'}V @VV{F}V\\
X_{G'} @>\pi_G >> X_G
\end{CD}\]


\noindent (The only non-trivial thing to show here is that $F_q$ acts trivially on $X_G$, a fact which follows from the Springer correspondence (see also \cite[p. 24]{GeMa}).) It follows that our parametrisation of geometric unipotent classes respects the $F$-action, and thus we have a $q$-independent parametrisation of $F$-stable unipotent classes of $G$.

We now turn our attention to the unipotent classes of $G^F$. For this we will need to assume that $X/\Z\Phi$ is torsion-free. Let $C$ be an $F$-stable unipotent class in $G$ and fix an $F$-stable point $u \in C$. Then, by the Lang-Steinberg theorem, the $G^F$-orbits in $C^F$ are parametrised by the $F$-conjugacy classes of $A(u) = C_G(u)/C^{\circ}_G(u)$. Explicitly, this is done as follows: For each $F$-conjugacy class in $A(u)$, choose a representative $a$; then choose $g_a \in $G such that $g_a^{-1}F(g_a) = \dot{a}$ for some representative $\dot{a}$ of $a$ in $C_G(u)$; then $\{ g_aug_a^{-1} \}$ is a set of representatives for the $G^F$-conjugacy classes in $C^F$.

The next step is to make a special choice for $u$ for which the $F$-action on $A(u)$ is as simple as possible and so that we have a more canonical reference point in $C$ for when we vary $q$. Thankfully, this is possible by an idea of T. Shoji. If $G$ is simple modulo its centre and $\Phi$ is not of type $E_8$ then $u$ may be chosen to be a so-called {\em split element} (proved in \cite{Sho5}, \cite{Sho9}, and \cite{BeSp}; see also the survey \cite{Sho3}, \S 5). We omit the definition here, but suffice to say that split elements in $C$ comprise a unique $G^F$-conjugacy class and, if $u$ is a split element, $F$ acts trivially on $A(u)$. Thus, the $G^F$ classes in $C^F$ correspond canonically to the conjugacy classes of $A(u)$. In fact, suppose $u' \in G'$ is a unipotent element whose $G'$-class agrees with the $G$-class of $u$ under Spaltenstein's map, then we may write down an explicit bijection between the conjugacy classes of $A(u)$ and $A(u')$. (See \cite{McnSo}, p. 336, for the details of this bijection, although this is based on earlier work in \cite{Miz} and \cite{Ale}.) In this manner we can label the unipotent classes, and hence the generalised Gelfand-Graev characters, of $G^F$ independently of $q$.

We now consider groups of type $E_8$. Following \cite{Kawa3}, Proposition 1.2.1, we dichotomise the situation according to whether $q$ is congruent to $1$ or $-1$ modulo $3$. (Note that we do not consider the case where $q$ is a power of $3$ since this is a bad prime for $E_8$.) From now on, when we refer to ``polynomials in $q$'' or ``treating $q$ as a variable'' we tacitly assume that we have fixed one or the other of these situations. In fact we only really need this distinction when $u$ is in the geometric unipotent class with Bala-Carter label $E_8(b_6)$ (corresponding to Mizuno label $D_8(a_3)$). In this case there is a $G^F$-class of split elements if $q$ is congruent to $1$ modulo $3$ (and hence the $G^F$-classes may be parametrised independently of $q$ as before), but split elements do not exist in this class if $q$ is congruent to $-1$ modulo $3$. However, it is possible to deal with this case explicitly. (E.g. it is known that there are precisely three $G^F$-classes and that their class sizes are given by different polynomials so we could, for instance, label these classes by these known polynomials.) This distinction is also implicitly used in Section \ref{LusSec} since Green functions, which appear there, have an analogous $q$ dependence issue for groups of type $E_8$.





By a well-known process of reduction in the theory of unipotent classes of reductive groups we may also lift the assumption that $G$ be simple modulo its centre. (See the standard text of Carter \cite{Car2}, Ch. 5, for a general treatment, and \cite{GoRo}, p. 7, for a discussion relevant to the current context.)

In summary, we have the following:



\begin{proposition} \label{GGGRparam}

Fix a root datum $(X,Y,\Phi,\check{\Phi})$ and automorphism $F_0$ of the Dynkin diagram associated with $(X,Y,\Phi,\check{\Phi})$, and assume that $X/\Z \Phi$ is torsion-free. Then we may parametrise the unipotent classes (and therefore the generalized Gelfand-Graev characters) of all finite reductive groups which have this data independently of the associated prime power $q$, provided $q$ is a power of a good prime. 

\end{proposition}

We will also make use of the following, sometimes implicitly.

\begin{proposition} \label{centralisersarepoly}

With the set-up of Proposition \ref{GGGRparam}, let $R$ be a set of $q$-independent labels for the unipotent classes of these groups. For each power $q$ of a good prime and each $r \in R$, let $u_{r,q}$ be a representative of the corresponding unipotent class. Then, allowing $q$ to vary, the order of the centraliser of $u_{r,q}$ is a polynomial in $q$.

\end{proposition}

\begin{proof}

This is \cite{GoRo2}, Proposition 3.3. The proof appeals to the Lusztig-Shoji algorithm for computing Green functions. \end{proof}




\section{Type $A_n$: Kawanaka's formula}   \label{Kawanaka}

In this section we set $G=\GL_n(k)$ and endow it with a split or non-split $\F_q$-rational structure, with corresponding Frobenius endomorphism $F$, so that $G^F$ is $\GL_n(\F_q)$ or $\GU_n(\F_q)$.  Using \cite{Kawa}, (3.2.14), we prove the following:

\begin{theorem} \label{A_nThm}

Let $G$ and $F$ be as above, $u \in G^F$ a unipotent element, and $\Gamma_u$ the corresponding generalised Gelfand-Graev representation. Then the dimension of the endomorphism algebra $\End_{\bar{\Q}_lG^F} \Gamma_u$ is a monic polynomial in $q$ with rational coefficients. Moreover, its degree is given by the dimension of the centraliser $C_G(u)$. \end{theorem}



Note that we need no condition on $p$ here since all primes are good. Before we can state Kawanaka's formula we must first explain the ingredients from \cite{Kawa}. The unipotent classes of $G$ are parametrised by the partitions of $n$, via the Jordan normal form, and the rational points of such a class comprise a single $G^F$-class. We may therefore denote a typical generalised Gelfand-Graev character by $\gamma_{\lambda}$, where $\lambda \vdash n$. We may also use partitions to label the non-zero values of generalised Gelfand-Graev characters since they are known to vanish on non-unipotent elements. We shall therefore adopt the convention of writing $\gamma_{\lambda}(\mu)$ for the character value of $\gamma_{\lambda}$ on the class corresponding to $\mu$. We set $\e=1$ or $-1$ depending on whether $F$ is split or not, respectively. If $\lambda = (\lambda_1, \lambda_2, \dots, \lambda_r)\vdash n$ then we define \[n(\lambda) = \sum_i{(i-1)\lambda_i}.\]

With reference to a fixed rational maximal torus $T$, we may, by the Lang-Steinberg theorem, label the $G^F$-classes of rational maximal tori by the $F$-classes of the Weyl group $W = N_G(T)/T \cong \Sn_n$. If $T$ is the diagonal maximal torus then $F$ acts trivially on $W$ and therefore we may label these by the classes of $W$ and, thus, by the partitions of $n$. (Here we have chosen the Frobenius endomorphism defining $\GU_n(\F_q)$ to be \label{unitary1} $F(g) = {}^tF_q(g^{-1})$, where $g \in G$ and $F_q$ is a standard Frobenius endomorphism.) With this set-up we denote representatives of the $G^F$-classes of rational maximal tori by $T_{\lambda}$ for $\lambda \vdash n$, and define \[W_{\lambda} = (N_G(T_{\lambda})/T_{\lambda})^F.\]

For $\lambda = (\lambda_1, \lambda_2, \dots, \lambda_r) \vdash n$, we set \[\sgn_{\e}(\lambda) = \e^{\lfloor n/2 \rfloor}(-1)^{n+r},\]

\noindent and \[e_{\lambda}(t) = \prod_i {(1 - t^{\lambda_i})}.\]

$Q_{\lambda}^{\mu}(t) \in \Z[t]$ will denote the Green polynomial with parameters $\lambda, \mu \vdash n$ (cf. \cite{Mac}). Finally, we define a rational function \[X_{\lambda}^{\mu}(t) = t^{n(\mu)}Q_{\lambda}^{\mu}(t^{-1}).\]

\noindent In fact this is a polynomial of degree $n(\mu)$ by \cite{Mac}, Chapter III, \S 7. Now we may state Kawanaka's formula. 

\begin{theorem}    \label{Kawaformula}   {\em (  \cite{Kawa}, (3.2.14) )}  With the above set-up, \[ \gamma_{\mu}(\lambda) = \e^{n(\mu)}\sum_{\rho \vdash n}{ |W_{\rho}|^{-1}  \sgn_{\e}(\rho)  q^n  e_{\rho}((\e q)^{-1})  X_{\rho}^{\mu}(\e q) Q_{\rho}^{\lambda}(\e q)   }. \]

\end{theorem}

Recall that a Green function of $G^F$ is the restriction of a Deligne-Lusztig virtual character $R_{T,\theta}$ to $G_{\uni}^F$. In the case of general linear groups these are simply Green polynomials, and for unitary groups they are of the form $Q_{\lambda}^{\mu}(-q)$. We shall need the following orthogonality formula for Green functions. (See, e.g., \cite{Sho3}.) Let $\lambda^F$ denote the unipotent class in $G^F$ corresponding to $\lambda \vdash n$. Then  \begin{equation} \label{Greenorthog} |G^F|^{-1} \sum_{\lambda \vdash n} {	|\lambda^F| Q_{\rho}^{\lambda}(\e q)  Q_{\pi}^{\lambda}(\e q)	=  \frac{ |N_G(T_{\rho}, T_{\pi} )^F|  }{|T_{\rho}^F||T_{\pi}^F|}},   \end{equation}

\noindent where $N_G(T_{\rho}, T_{\pi} )  =  \{  n\in G \, |\, n^{-1}T_{\rho}n = T_{\pi}  \}$.

Now $\dim \End_{\bar{\Q}_l G^F} \Gamma_{\mu}$ may be written as\begin{eqnarray} \label{Kawapoly1} \la \gamma_{\mu}, \gamma_{\mu} \ra &=& |G^F|^{-1}  \sum_{\lambda \vdash n}|\lambda^F| \left(\sum_{\rho \vdash n} |W_{\rho}|^{-1} \sgn_{\e}(\rho) q^n e_{\rho}((\e q)^{-1})X_{\rho}^{\mu}(\e q)Q_{\rho}^{\lambda}(\e q)\right) \nonumber \\
&&\times \left(\sum_{\pi \vdash n} |W_{\pi}|^{-1} \sgn_{\e}(\pi) q^n e_{\pi}((\e q)^{-1}) X_{\pi}^{\mu}(\e q) Q_{\pi}^{\lambda}(\e q)\right),   \nonumber    \end{eqnarray} 

\noindent which is a polynomial in $q$. Indeed, it is easy to check that it is a rational function of the form $f/|G^F|$ for some $f \in \Q[q]$. (We consider $|G^F|$ as a polynomial in $q$ here.) So $f(q)/|G^F| \in \Z$ for infinitely many $q$. By applying the division algorithm we may write $f = g|G^F| + r$, where $\deg r < \deg |G^F|$. It follows that for some integer $c$, $cr(q)/|G^F| \in \Z$ for all $q$. But the limit as $q\rightarrow \infty$ is $0$, so $r=0$ and the claim follows. We label this argument by $\clubsuit$ as we shall reuse it later. The fact that conjugacy class sizes are given by polynomials in $q$ can be deduced by applying $\clubsuit$ and the orbit-stabiliser theorem to Proposition \ref{centralisersarepoly}.

Now we will show that the degree of this polynomial is $\dim C_G(u)$, where $u$ is in the class corresponding to $\mu$ by the Jordan normal form. For $\GL_n(k)$ it is known that $\dim C_G(u) = n + 2n(\mu)$. (Adapt, e.g., the corresponding result, \cite{Geck}, Proposition 2.6.1, for $\SL_n(k)$.) To make the derivation tidier we define an equivalence relation on $\Q[q]$, denoted $\approx$, by setting $f \approx g$ if $\deg f = \deg g$, for $f, g \in \Q[q]$. Under this relation the above expression is equivalent to \begin{eqnarray} \label{Kawapoly2} && q^{-n^2 + n + 2n(\mu)} \sum_{\lambda \vdash n}|\lambda^F| \left(\sum_{\rho \vdash n} \frac{q^n e_{\rho}((\e q)^{-1})(-1)^{r(\rho)}}{|W_{\rho}|} Q_{\rho}^{\lambda}(\e q)\right)  \nonumber \\
& & \times\left(\sum_{\pi \vdash n} \frac{(-1)^{r(\pi)}}{|W_{\pi}|} Q_{\pi}^{\lambda}(\e q)\right) \nonumber \\
& \approx & q^{-n^2 + n + 2n(\mu)} \sum_{\rho, \pi \vdash n} \frac{q^n e_{\rho}((\e q)^{-1})(-1)^{r(\rho) + r(\pi)}}{|W_{\rho}||W_{\pi}|}   \sum_{\lambda \vdash n}|\lambda^F|  Q_{\pi}^{\lambda}(\e q)   Q_{\rho}^{\lambda}(\e q)  \nonumber \\
& \approx & q^{-n^2 + n + 2n(\mu)} \sum_{\rho \vdash n} \frac{q^n e_{\rho}((\e q)^{-1})|G^F|}{|W_{\rho}||T_{\rho}^F|}  \ \ \  \mbox{ (by (\ref{Greenorthog})) }\nonumber \\ 
& \approx & q^{n + 2n(\mu)} \sum_{\rho \vdash n} \frac{1}{|W_{\rho}|} \ \  \approx \ \  q^{n + 2n(\mu)}, \nonumber \end{eqnarray}

\noindent where we have used the fact that $|T_{\rho}^F| = q^n e_{\rho}((\e q)^{-1})$. So, to complete the proof of Theorem \ref{A_nThm}, it remains to show that $\dim \End_{\bar{\Q}_l G^F} \Gamma_{\mu}$ is monic. Indeed, observe that the coefficient of the leading term has been preserved in the above until the last line. Using the fact (see, e.g., \cite{Kawa}) that $W_{\rho} \cong C_{\Sn_n}(\rho)$, the centraliser in the symmetric group of an element of cycle type $\rho$, we see that \begin{equation}  \label{one}    \sum_{\rho \vdash n} \frac{1}{|W_{\rho}|} = \sum_{\rho \vdash n} \frac{|\cl(\rho)|}{n!} = 1,\end{equation}

\noindent which completes the proof of Theorem \ref{A_nThm}.

\section{The general case: Lusztig's formula}  \label{LusSec}

Inspired by Kawanaka's work, Lusztig (\cite{Lus}, Theorem 7.3)) has derived a similar formula to that in Theorem \ref{Kawaformula}, valid for any connected reductive group, but assuming that $p$ is large. (We assume that $p$ is large enough in the sense of \cite{Lus} throughout this section.) Lusztig's formula, however, is rather more geometric and is given in terms of intersection cohomology complexes of closures of unipotent classes with coefficients in various local systems. We will use it to prove the following.

\begin{theorem} \label{GenThm}

Let $G$ be a connected reductive group, defined over $\F_q$, with root datum $(X,Y,\Phi,\check{\Phi})$ and Frobenius endomorphism $F$. Let $u\in G^F$ be a unipotent element and let $\Gamma_u$ be the corresponding generalised Gelfand-Graev representation. Then, assuming that $X/\Z\Phi$ is torsion-free, the dimension of the endomorphism algebra $\End_{\bar{\Q}_lG^F} \Gamma_u$ is a monic polynomial in $q$ with rational coefficients. Moreover, its degree is given by the dimension of the centraliser $C_G(u)$. \end{theorem}

In \cite{Lus} Lusztig has associated a generalised Gelfand-Graev representation $\Gamma_{\phi}$ to each homomorphism $\phi : \s\l_2 \rightarrow \gl$, where $\gl = \Lie G$. However, for convenience we will use the equivalent notation $\Gamma_N = \Gamma_{\phi}$, where $N \in \gl$ is the nilpotent element which is the image under $\phi$ of the matrix with $1$ in the $(1,2)$-position and $0$ elsewhere. This is defined on certain rational points of the nilpotent variety $\gl_{\nil}$ of $\gl$, but can be made into a usual generalised Gelfand-Graev representation as follows. For any Frobenius endomorphism on $G$ there exists an $\Ad G$-compatible Frobenius endomorphism on $\gl$. (We denote the Frobenius endomorphism on $\gl$ also by $F$. Hence, the domain of $\Gamma_N$ is $\gl_{\nil}^F$.) Furthermore, by \cite{SpSt}, Theorem 3.2, there exists a Springer morphism (i.e. a $G$-equivariant bijective morphism of varieties) $f : G_{\uni} \rightarrow \gl_{\nil}$ which is compatible with these Frobenius endomorphisms. For the purposes of this article, then, we can and will identify unipotent and nilpotent elements via $f$. Then for $u \in G_{\uni}^F$, $x \in G^F$, \[\Gamma_{u}(x) =  \left\{ \begin{array}{rl}
 \Gamma_{f(u)} \circ f(x) &\mbox{ if $x\in G_{\uni}^F$}, \\
  0                        &\mbox{ otherwise}. \end{array} \right.\]

The basic parameter set used in \cite{Lus} is the set $\textgoth{I}$ of all pairs $(C, \textgoth{F})$, where $C$ is a nilpotent orbit in $\gl$ and $\textgoth{F}$ is an irreducible, $G$-equivariant, $\bar{\Q}_l$-local system on $C$, up to isomorphism. By starting with a closed subgroup $L$ of $G$, which is the Levi subgroup of some parabolic subgroup of $G$, we may obtain another parameter set in the same manner. In this way, we obtain triples $(L,c,\textgoth{L})$, where $c$ is a nilpotent $\Ad(L)$-orbit in $\l = \Lie(L)$, and \textgoth{L} is an irreducible, $L$-equivariant, $\bar{\Q}_l$-local system on $c$, given up to isomorphism. Using the generalised Springer correspondence we can partition the set $\textgoth{I}$ into blocks, and to each block we may associate (the $G$-orbit of) such a triple (a {\em cuspidal} triple).

By \cite{Lus}, \S 4, the elements of these blocks are naturally parametrised by the irreducible characters of the Weyl group $W_L = N_G(L)/L$. There is a single block associated with the maximal tori, since all maximal tori are $G$-conjugate and contain only the trivial unipotent class. We shall call this block the {\em principal block}, by analogy with Harish-Chandra theory.

In each class $C$ fix a representative $u$ once and for all, and consider the component group $A(u) = C_G(u)/C^{\circ}_G(u)$. This acts naturally on the stalk $\textgoth{F}_u$ of a $G$-equivariant local system $\textgoth{F}$ on $C$, and thus gives rise to a finite dimensional $\bar{\Q}_l$-representation of $A(u)$. On the other hand, if $\rho \in \Irr A(u)$ then we may obtain the irreducible $G$-equivariant local system $\Hom_{A(u)} (\rho, \pi_*\bar{\Q}_l)$, where $\pi : G/C_G(u)^{\circ} \rightarrow G/C_G(u) \cong C$ is a finite \`{e}tale covering with group $A(u)$ (cf. \cite{Sho8}, p. 74). We shall denote by $\mathcal{N}^G$ the set of all pairs $(C, \psi)$, where $C$ is a nilpotent orbit in $\gl$ and $\psi \in \Irr A(u)$. By the above $\textgoth{I}$ may be naturally identified with $\mathcal{N}^G$.


Assume now that $G$ is defined over $\F_q$, with Frobenius endomorphism $F$. Then $F$ acts on $\textgoth{I}$ by \[F: (C, \textgoth{F}) \rightarrow (F^{-1}(C), F^*(\textgoth{F})),\]

\noindent where $F^*(\textgoth{F})$ is an inverse image of local systems. If $C$ is $F$-stable and $u \in C^F$ then the $F$-action respects the natural identification of $\textgoth{I}$ and $\mathcal{N}^G$. If $C$ is not $F$-stable, more care is needed to describe the action. (But it is still respected by $F$.) The correspondence between blocks and triples $(L,c,\textgoth{L})$ also respects the $F$-action (cf. \cite{Lus3}, 24.2). Thus, $F$ permutes the blocks. In fact only the $F$-stable blocks feature in Lusztig's character formula, so we will not be interested in blocks which are not $F$-stable. (Note that the principal block is always $F$-stable.)


Following the natural parametrisation of the elements of a block $\textgoth{I}_0$ (associated with a triple $(L,c,\textgoth{L})$) by the irreducible characters of $W_L$, we set $\Irr W = \{ V_{\i}\ |\ \i \in \textgoth{I}_0 \}$ (the $V_{\i}$ being modules). For $\i = (C,\textgoth{F}) \in \textgoth{I}_0$ we define $\support(\i) = C$. With respect to a fixed rational Levi subgroup $L$, we may parametrise the $G^F$-orbits of the rational Levi subgroups which are $G$-conjugate to $L$ by the $F$-classes of $W_L$, using the Lang-Steinberg theorem, in the same manner as for maximal tori. We let $\mathcal{Z}_{L_w}$ denote the centre of the Levi subgroup $L_w$ corresponding to the $F$-class of $w \in W_L$. (In fact $\mathcal{Z}_{L_w}$ is connected if $Z(G)$ is by \cite{DiMi}, Lemma 13.14.) We also set $\t$ to be the centre of $\Lie(L)$.

 


\begin{theorem} {\em (\cite{Lus}, Theorem 7.3)}  \label{Lusthm}  Let $G$ be a connected reductive group with a split $\F_q$-structure, and let $\textgoth{I}_0$ be an $F$-stable block and let $(\gamma_{N})_{\textgoth{I}_0}: \gl^F \rightarrow \bar{\Q}_l$ be the function defined by \[\sum_{\i, \i', \i_1 \in \textgoth{I}_0} q^{f'(\i,\i_1)} \zeta^{-1} |W_L|^{-1} \sum_{w \in W_L} \Tr(w, V_{\i}) \Tr(w, V_{\hat{\i}_1}) |{\mathcal{Z}^{\circ F}_{L_w}}| {P'}_{\i', \i} \ov{ \mathcal{Y}_{\i'}(-N')} \mathcal{X}_{\i_1},\]

\noindent where \begin{eqnarray} f'(\i, \i_1) &=& -\dim \support(\i_1)/2 + \dim \support(\i)/2 \nonumber \\
&& - (\dim \Ad(G)N)/2 + (\dim \gl / \t)/2,   \nonumber    \end{eqnarray}

\noindent $\zeta$ is a certain fourth root of 1 and $\i \mapsto \hat{\i}$ is a certain bijection $\textgoth{I}_0 \rightarrow \textgoth{I}_0$ {\em (}both defined in {\em \cite{Lus} )}. Then \[\gamma_{N} = \sum_{\textgoth{I}_0} (\gamma_{N})_{\textgoth{I}_0},\]

\noindent where $\textgoth{I}_0$ runs over the set of all $F$-stable blocks. \end{theorem}

\begin{rmk} The functions $\mathcal{X}_{\i}, \mathcal{Y}_{\i}$ are analogous to the Green polynomials in Theorem \ref{Kawaformula} (in fact, related to generalised Green functions). These are certain nilpotently supported functions $\gl^F \rightarrow \bar{\Q}_l$, and the $P'_{\i',\i}$ are related combinatorial objects (cf. \cite{Lus}, \S \S 6.4 -- 6.6). Much is known about these in the special case that $\textgoth{I}_0$ is the principal block and we shall exploit this information in the course of the proof of Theorem \ref{GenThm}. \end{rmk}

\subsection{The proof of Theorem \ref{GenThm}}  \label{proof}

We will now prove Theorem \ref{GenThm} under the assumption that $G$ has a split $\F_q$ structure. Since this is equivalent to $F_0$ acting trivially on $\pi$, it follows that all geometric unipotent classes are $F$-stable. We show how to remove this assumption in the next subsection.

In addition to the set-up of Section \ref{intro} we must show that the $F$-stable blocks may be parametrised independently of $q$, in order to establish a rigorous foundation for the proof of Theorem \ref{GenThm}. Before we consider blocks, however, we first describe a treatment of Levi subgroups which is independent of $q$. Recall that, with respect to fixed data $(X,Y,\Phi,\check{\Phi}), F_0$, we have fixed a choice of maximally split maximal torus for each prime power $q$, and that we have fixed a simple system $\pi \subset \Phi$ (such that $F_0(\pi) = \pi$) so that a rational Borel subgroup $B \supset T$ is determined. For each subset $J \subset \pi$ such that $F_0(J) = J$ we let $P_J$ denote the corresponding standard parabolic subgroup containing $B$, and $L_J$ the unique Levi subgroup of $P_J$ containing $T$. Since both $T$ and $B$ are $F$-stable, so are $P_J$ and $L_J$. As mentioned above the $G^F$-orbits of $F$-stable Levi subgroups conjugate to $L_J$ are parametrised by the $F$-classes of $W_{L_J}$. We have thus parametrised all $F$-stable Levi subgroups of $G$ (up to $G^F$-action) independently of $q$.

Now we move on to consider blocks. As mentioned a block is $F$-stable precisely when the corresponding triple $(L,c,\textgoth{L})$ is $F$-stable. I.e. its image under $F$ is in the same $G$-orbit. Any such $L$ is $G$-conjugate to some $L_J$ so we may assume that our triple is $(L_J,c,\textgoth{L})$. Since the $F$-action on unipotent classes is independent of $q$, the same is true of nilpotent orbits, since Springer morphisms respect the $F$-action. So, in order to parametrise the $F$-stable blocks independently of $q$ it just remains to parametrise the irreducible, $L$-equivariant, $\bar{\Q}_l$-local systems independently of $q$. But, as mentioned earlier, these are naturally parametrised by the elements of $\Irr A(u)$, for any $u$ such that $f(u) \in c$. If we choose $u$ to be a split element then we can thus obtain such a $q$-independent parametrisation, analogous to the parametrisation of unipotent classes of $G^F$ considered earlier. In the special case that split elements do not exist we may, by our previous discussion in Section \ref{intro}, choose $u$ to be from a $G^F$-class corresponding to a fixed $q$-independent label, which is sufficient for the current task. Note that in all cases $F$ acts trivially on $A(u)$. (For $u$ split this is clear; for the other case see \cite{Kawa3}, 1.2.1), and thus preserves the isomorphism classes of the corresponding local systems. In summary we have the following, which holds for split or non-split groups.

\begin{lemma} \label{blocksqindep}

The $F$-stable blocks are parametrised independently of $q$.\end{lemma}





For the remainder of this subsection we assume that $G$ has a split $\F_q$-structure. Define, for $\i, \i'$ in the same $F$-stable block, \[\omega_{\i, \i'} = |W_L|^{-1} q^{-\codim C/2 -\codim C'/2 + \dim \t}\sum_{w \in W_L} \Tr(w, V_{\i}) \Tr(w, V_{\i'}) \frac{|G^F|}{|\mathcal{Z}^{\circ F}_{L_w}|}\]

\noindent where $C = \support(\i)$ and $C' = \support(\i')$. Set $\omega_{\i, \i'} = 0$ if $\i, \i'$ are in different blocks (cf. \cite{Lus}, \S 6.5). Also define \[\al_{\i, \i'} = \sum_{w \in W_L} \Tr(w, V_{\i}) \Tr(w, V_{\i'}) |\mathcal{Z}^{\circ F}_{L_w}|.\]

One may check, using the relations from \cite{Lus}, \S \S 6.5, 6.6, that \[ \sum_{x \in \gl_{\nil}^F}  \ov{\mathcal{X}_{\i}(x)}	\mathcal{X}_{\i'}(x) = \omega_{\i, \i'}.\]    

\noindent It follows that the class functions $(\gamma_{N})_{\textgoth{I}_0}$ are mutually orthogonal. We may therefore write \[ \la \gamma_{N}, \gamma_{N} \ra = \sum_{\textgoth{I}_0} \left\la (\gamma_{N})_{\textgoth{I}_0}, (\gamma_{N})_{\textgoth{I}_0} \right\ra,\]

\noindent summing over the $F$-stable blocks. Each summand may be written as follows. \[  |G^F|^{-1}   \sum_{\i, \i_1, j, j_1 \in \textgoth{I}_0}  q^{f'(\i, \i_1)  +  f'(j, j_1)}   |W_L|^{-2}   \ov{\al_{\i, \hat{\i}_1}}   \al_{j, \hat{j}_1}   \omega_{\i_1, j_1} \]   \begin{equation} \label{notprincipal} \times \sum_{\i', j' \in \textgoth{I}_0}    {P'}_{\i', \i} \ov{ \mathcal{Y}_{\i'}(-N')       }  {P'}_{j', j} \mathcal{Y}_{j'}(-N'). \end{equation}

From now on we will assume that $X/\Z \Phi$ is torsion-free so that we may begin talking about polynomials in $q$ as in Section \ref{intro}. $|\mathcal{Z}^{\circ F}_{L_w}|$ can be seen to be a polynomial in $q$ by, e.g., \cite[p.74]{Car2}.


\begin{lemma} \label{qpoly}   Under the assumptions of Theorem \ref{GenThm}, $\dim\End_{\bar{\Q}_lG^F} \Gamma_N$ is a polynomial in $q$ with rational coefficients, and, for each block $\textgoth{I}_0$, \[|G^F|  \left\la (\gamma_{N})_{\textgoth{I}_0}, (\gamma_{N})_{\textgoth{I}_0} \right\ra \] \noindent is a Laurent polynomial in $q$ with rational coefficients. \end{lemma}

\begin{proof}

By Lemma \ref{blocksqindep} we may consider the contribution from each $F$-stable block independently. The second statement is clear with respect to the top line of (\ref{notprincipal}) if the $\omega_{\i, j}$ are polynomials in $q$ for any $\i, j \in \textgoth{I}_0$. But this follows from Argument $\clubsuit$, used with the fact (cf. \cite{Lus}, \S 6.5) that the $\omega_{\i, j}$ are integers. The corresponding statement for the $P_{\i, j}'$ then follows from the fact that they are defined (\cite{Lus}, p. 151) in terms of the $\omega_{\i, j}$. Now let $\i = (C, \textgoth{F}) \in \textgoth{I}_0$ and assume (as we may, since $G$ has a split $\F_q$-structure) that $C$ is $F$-stable, and let $u \in C^F$. Then $\mathcal{Y}_{\i}(u)$ is defined (cf. \cite{Lus}, \S 6) in terms of the action of $A(u)$ on the stalk $\textgoth{F}_u$ at $u$, and a certain choice of scalar multiple. Subject to the conventions of Section \ref{intro} $\textgoth{F}_u$ and the action of $A(u)$ on it are independent of $q$. So it remains to check that this scalar multiple can be chosen independently of $q$. One way of seeing this is via \cite{Geck2}, (2.2), where the scalar is uniquely determined by a choice of extension of the character of $A(u)$ corresponding to $\textgoth{F}$ to the semidirect product $A(u)\la F \ra$. Hence, the scalar can be chosen independently of $q$. The second statement then follows by applying Argument $\clubsuit$.\end{proof}


Define the degree of a Laurent polynomial $\sum \al_i t^i$ to be the largest integer $i$ such that $\al_i \not= 0$. Then by Lemma \ref{qpoly}, in order to prove Theorem \ref{GenThm} it is sufficient to consider   \[\deg \left( |G^F| \left\la (\gamma_{N})_{\textgoth{I}_0}, (\gamma_{N})_{\textgoth{I}_0} \right\ra \right)\]

\noindent for the various blocks.  The following lemma describes some properties of the degrees of polynomials involved in (\ref{notprincipal}) in the case that $\textgoth{I}_0$ is the principal block.

\begin{lemma}  \label{principallemma}

Let $\i, \i', \i_0$ belong to the principal block, with $\support(\i_0)$ equal to the regular nilpotent orbit, and let $n = \rank G$. Then the following hold.

\begin{enumerate}[{\em (i)}]

  \item $\i_0$ is unique with this property.
  
	\item $\deg \omega_{\i,\i'} \le (\dim \support(\i) + \dim \support(\i'))/2$, with equality if $\i = \i'$.
	
	\item $\deg \al_{\i,\i'} \le n$, with equality if, and only if,  $\i = \i'$.
	
	\item $\sum_{\i' \in \textgoth{I}_0}   {P'}_{\i', \i} \mathcal{Y}_{\i'}(-N') = 1$ if $\i=\i_0$.  \label{Ysum}
	
	\item $f'(\i',\i) \le f'(\i_0,\i)$, with equality if, and only if, $\i' = \i_0$.  \end{enumerate}   \end{lemma}

\begin{proof}

If $u$ is regular then $A(u) = 1$ since $C_G(u)$ is connected (\cite{Car2}, Proposition 5.1.6). This shows that the regular nilpotent class can only appear in one element of $\textgoth{I}$. Moreover, it must be in the principal block, by the Springer correspondence (cf.  \cite{Car2}, \S 12.6). This proves (i). (ii) is equivalent to showing that \begin{equation}  \label{torusineq}  \deg \sum_{w \in W_L} \Tr(w, V_{\i}) \Tr(w, V_{\i'}) \frac{|G^F|}{|\mathcal{Z}^{\circ F}_{L_w}|} \le \dim G - n,\end{equation}

\noindent with equality if $\i = \i'$.  Since we are in the principal block, $\mathcal{Z}^{\circ F}_{L_w} = T^F_w$, and its order is always a polynomial in $q$ of degree $n$. (See, e.g., \cite{Car2}, Chapter 2.) Then the required statement is an elementary exercise in character theory and properties of sums of rational functions. (iii) follows by a similar argument.

Since $\textgoth{I}_0$ is the principal block we may use the theory of Green functions (as opposed to generalised Green functions), and for this we refer to \cite{Sho3}, \S 5. We may define related class functions associated with irreducible characters of the Weyl group as follows. For $\chi \in \Irr(W)$, let \[ Q_{\chi} = |W|^{-1} \sum_{w \in W} \chi (w) Q_{T_w}.\]

\noindent Then, for $\i \in \textgoth{I}_0$, the principal block, we have \[Q_{\i}  =  \sum_{j \in \textgoth{I}_0}  P_{j, \i} \mathcal{Y}_j,\]


\noindent where $\chi \leftrightarrow \i$ is the Springer correspondence. Since the element of $\Irr (W)$ corresponding to $\i_0$ by the Springer correspondence is the trivial character, the expression in (iv) may be written as $Q_1(-N')(q^{-1})$. But $Q_1 = 1$ by, e.g., \cite{DiMi}, Proposition 12.13.

(v) follows from the fact that the dimension of the regular nilpotent orbit is strictly greater than that of the others. \end{proof}

We may now deduce the following.

\begin{proposition} \label{principaldegree} Let $\textgoth{I}_0$ be the principal block. Then \[ \deg \left( |G^F| \left\la (\gamma_{N})_{\textgoth{I}_0}, (\gamma_{N})_{\textgoth{I}_0} \right\ra \right) = \dim G + \dim C_G(u).\] \end{proposition}

\begin{proof}

We will show that the required degree is obtained by a careful choice of parameters $\i, \i_1, j, j_1$. We choose $\i = j = \i_0$, the unique parameter corresponding to the regular orbit, and then choose $\i_1 = j_1$ such that $\hat{\i}_1 = \i$. (This is a unique choice since $\ \hat{}\ $ is a bijection.) With this fixed choice we see, by Lemma \ref{principallemma}, that the degree is given by \[ 2 f'(\i_0, \i_1)  +  \deg \al_{\i_0, \hat{\i}_1}  + \deg \al_{\i_0, \hat{\i}_1}   +   \deg\omega_{\i_1, \i_1}\]\[= \dim G     +    \dim \support(\i_0) +n - \dim \Ad(G)N.\]


\noindent It is known that the regular orbit has dimension $\dim G - n$. It follows that we obtain the  required degree. To complete the proof, note that by Lemma \ref{principallemma}, any deviation from this choice of $\i, \i_1, j, j_1$ gives rise to a polynomial of strictly lower degree. \end{proof}

Furthermore, one can check (using (\ref{notprincipal})) that $|G^F| \la (\gamma_{N})_{\textgoth{I}_0}, (\gamma_{N})_{\textgoth{I}_0} \ra$ is monic. But this is not sufficient to prove Theorem \ref{GenThm} since it may be the case that the leading terms from non-principal blocks exceed or annihilate this contribution. We shall, however, show that this is impossible, by describing some of the features of non-principal blocks. The associated Levi subgroups will no longer be maximal tori and so we have \begin{equation}   \label{rankk} \deg |\mathcal{Z}^{\circ F}_{L_w}| = \rank L_w = \dim \t =: m < n.\end{equation}

\begin{lemma}  \label{npblocklemma}

Let $\i, \i'$ belong to the non-principal block $\textgoth{I}_0$. Then the following hold.

\begin{enumerate}[{\em (i)}]

  \item No element of $\textgoth{I}_0$ is supported by the regular orbit.
  
	\item $\deg \omega_{\i,\i'} \le (\dim \support(\i) + \dim \support(\i'))/2$, with equality if $\i = \i'$.
	
	\item $\deg \al_{\i,\i'} \le m$, with equality if, and only if  $\i = \i'$.  \end{enumerate}   \end{lemma}

\begin{proof}

This follows from (\ref{rankk}) and the proof of Lemma \ref{principallemma}.\end{proof}

We may now deduce, for a non-principal block $\textgoth{I}_0$, that the degree of a typical term of \[|G^F| \left\la (\gamma_{N})_{\textgoth{I}_0}, (\gamma_{N})_{\textgoth{I}_0} \right\ra\] \noindent with respect to (\ref{notprincipal}) is not greater than \[f'(\i, \i_1)  +  f'(j, j_1)  +  \deg \al_{\i, \hat{\i}_1}  + \deg \al_{j, \hat{\i}_1}   +   \deg\omega_{\i_1, j_1},\]

\noindent but this is strictly less than $\dim G + \dim C_G(u)$. This completes the proof of Theorem \ref{GenThm}.

\subsection{Groups with a non-split $\F_q$-structure}

In the previous subsection the analysis of (\ref{notprincipal}) was made easier by the fact that the geometric unipotent classes were fixed by the Frobenius action. (In particular it was possible to deduce useful information about the $\mathcal{Y}_{\i}$.) For non-split groups we may reduce to a situation where the only functions $\mathcal{Y}_{\i}$ that appear are such that $\support(\i)$ is $F$-stable. We do this by considering a transposed version of generalised Gelfand-Graev characters as follows (see \cite{Lus}, \S 7.5) . 

Let $C$ be an $F$-stable nilpotent orbit such that $\i = (C, \textgoth{F})$ belong to an $F$-stable block $\textgoth{I}_0$, and let $x_1, \dots, x_r$ be representatives for the $G^F$-classes in $C^F$. Let $a$ denote the order of $C_{G^F}(x_i)$, for some $i$. (Clearly $a$ is independent of $i$.) Also, let $a_i$ denote the order of $C_{G^F}(x_i)^F$. Then we define \[\gamma_{\i} = \sum_{i=1}^r aa_i^{-1} \mathcal{Y}_{\i} (x_i) \gamma_{x_i}.\]

We also have \[\gamma_{x_i} = a^{-1}\sum_{k} \ov{\mathcal{Y}_{k} (x_i)} \gamma_{k},\]

\noindent where the sum is taken over the $F$-fixed $k$ such that $\support(k) = C$. Hence we may write \begin{equation} \label{basischange} \left\la \gamma_{x_i}, \gamma_{x_i} \right\ra = a^{-2} \sum_{k,l} \ov{\mathcal{Y}_{k} (x_i)} \mathcal{Y}_{l} (x_i) \left\la \gamma_{k}, \gamma_{l} \right\ra.\end{equation}

In \cite{Lus} Lusztig has given a formula for the $\gamma_{\i}$, valid for split groups, which is analogous to the one in Theorem \ref{Lusthm}. Lusztig also hints at how to alter various formulas contained in \cite{Lus} to make them valid for non-split groups. This is carried out explicitly in \cite{DLM}, from which we borrow the following formula (\cite{DLM}, (6.1)): \[ \gamma_{\i} = \sum_{\i, \i_1 \in \textgoth{I}_0} q^{f'(\i,\i_1)} \zeta^{-1} a |W_L|^{-1} \sum_{w \in W_L} \Tr(wF, \tilde{V}_{\i}) \Tr(wF, \tilde{V}_{\hat{\i}_1}) |{\mathcal{Z}^{\circ wF}_{L}}| {P'}_{\i_0, \i} \e_{\i_1}\mathcal{X}_{\i_1}\]

Here, $\tilde{V}_{\i}$ are certain extensions of $V_{\i}$ to modules for the group $W_L\la F \ra$. In \cite{Lus4}, \S 2.2, they are chosen in an explicit and unique way, which is independent of $q$. Lusztig calls them the {\em preferred extensions}. By \cite{DLM}, Remark 3.6, the scalar $\e_{\i_1}= \pm 1$ is determined by a preferred extension and thus is also independent of $q$.

Thus, we have \[\left\la \gamma_{k}, \gamma_{l} \right\ra =   |G^F|^{-1}  \sum_{\i, \i_1, j, j_1 \in \textgoth{I}_0}  q^{f'(\i, \i_1)  +  f'(j, j_1)}   a^2 |W_L|^{-2}   \ov{\al_{\i, \hat{\i}_1}}   \al_{j, \hat{j}_1}   \omega_{\i_1, j_1} \]   \begin{equation} \label{notprincipalnonsplit} \times  \ov{\e_{\i_1}} {P'}_{k, \i} \e_{j_1} {P'}_{l, j}, \end{equation}

\noindent where \[\al_{\i, j} = \sum_{w \in W_L} \Tr(wF, \tilde{V}_{\i}) \Tr(wF, \tilde{V}_{\hat{\i}_1}) |{\mathcal{Z}^{\circ wF}_{L}}|\] \noindent this time.

By combining (\ref{basischange}) and (\ref{notprincipalnonsplit}) we obtain an expression for the dimension of a generalised Gelfand-Graev representation in which no function $\mathcal{Y}_{\i}$ appears unless $\support(\i)$ is $F$-stable. Theorem \ref{GenThm} now follows for non-split groups by applying the same argument as in Subsection \ref{proof} to this expression.




\subsection{Groups with a disconnected centre}

In this subsection we consider what happens when we remove the assumption that $X/\Z \Phi$ is torsion-free. In this situation $Z(G)$ may have a disconnected centre. (In fact this occurs for $G=G(p^m)$ precisely when $X/\Z \Phi$ has no $p'$-torsion.) It should be clear from the discussion in Section \ref{intro} that the parametrisation of $F$-stable geometric unipotent classes is still independent of $q$. The difficulty arises when we try to parametrise, independently of $q$, the $G^F$-classes of rational points in some $C^F$, where $C$ is an $F$-stable geometric unipotent class. In fact this is impossible in general since even the number of $G^F$ classes in $C^F$ may depend on $q$. We now explain a way of getting around this difficulty.


First note that split elements still exist and are unique  up to $G^F$-conjugacy. (We may ignore the $E_8$ difficulty this time as there is only one isogeny class associated with this group and it has the property that $X/\Z \Phi$ is torsion-free.) For a given finite reductive group $G^F$, with data $(X,Y,\Phi,\check{\Phi}), F_0$ and prime power $q_0$, set $D = D(G^F)$ to be the set of prime powers $q_1$ such that the component groups of $F$-stable geometric classes of $G(q_1)$ are isomorphic to those of $G(q_0)$ (via Spaltenstein's map), and furthermore that the $F$-action respects these isomorphisms. Then if $f \in \Q[x]$ is a polynomial such that some quantity associated with $(X,Y,\Phi,\check{\Phi}), F_0$ is given by $f(q)$ for all $q \in D$, then we say that this quantity is a {\em polynomial in $q$ on $D$}.

For given data $(X,Y,\Phi,\check{\Phi}), F_0$ it should be possible to write down an explicit list of the possible sets $D$, for simple simply connected groups at least. The discussion on \cite{GoRo}, p. 8, would be a good starting point. We briefly illustrate what might happen by means of an example (\cite{GoRo}, Example 2.6). When $G=\SL_3$, with standard Frobenius map, and $u \in C$, where $C$ is the regular unipotent class, there are three possibilities for $D$, depending on the congruence of $q$ modulo $3$. When $q$ is a power of $3$, $A(u) = 1$. When $q$ is congruent to $1$ modulo $3$, $A(u) \cong \Z/3\Z$, with $F$ acting trivially. When $q$ is congruent to $2$ modulo $3$, $A(u) \cong \Z/3\Z$, but this time $F$ acts non-trivially.

We now state and prove the main result of this subsection.

\begin{theorem} \label{DThm}

Let $G$ be a connected reductive group, defined over $\F_q$, with root datum $(X,Y,\Phi,\check{\Phi})$ and Frobenius endomorphism $F$, and let $D = D(G^F)$ be as above. Let $u\in G_{\uni}^F$ and let $\Gamma_u$ be the corresponding generalised Gelfand-Graev representation. Then the dimension of the endomorphism algebra $\End_{\bar{\Q}_lG^F} \Gamma_u$ is a monic polynomial in $q$ on $D$ with rational coefficients. Moreover, its degree is given by the dimension of the centraliser $C_G(u)$. \end{theorem}

\begin{proof}

The proof essentially works in the same way as for groups with a connected centre. We will assume here, for simplicity, that $G$ has a split $\F_q$-structure, although the non-split generalisation carries over to this case, as before. First note that the statement of the theorem is meaningful, in the sense that the generalised Gelfand-Graev characters are naturally parametrised independently of $q \in D$, by the above discussion. Next, note that Lemmas \ref{blocksqindep} and \ref{qpoly} still hold in this situation. Indeed, the set $\textgoth{I}$ and its partition into blocks, as well as the $F$-action on it, only depends on $(X,Y,\Phi,\check{\Phi})$, $F_0$, and the component groups, together with their $F$-action. The analysis of (\ref{notprincipal}) required in the proof of Lemma \ref{qpoly} poses no new difficulty, although the reason why $|\mathcal{Z}^{\circ F}_{L_w}|$ is a polynomial in $q$ is somewhat deeper in this case (see, e.g., \cite{Car2}, pp.73,74).

The statement of Lemma \ref{principallemma} remains true here also, although establishing the truth of (i) requires a different argument since it may not be the case that $A(u)=1$ for a regular unipotent element any more. However, (i) is equivalent to there only being one Springer representation associated with the regular nilpotent orbit. But this is clear from Springer's construction (as described in, e.g., \cite{Car2}, \S 12.6). This fact also means that the special parameters chosen in the proof of Lemma \ref{principaldegree} may be chosen again in the current situation and thus this result, and the monic property, remain true too. Finally, Lemma \ref{npblocklemma}, (i) may not be true in the current situation, but (ii) and (iii) are true and these are, in fact, sufficient to deduce that polynomials associated with non-principal blocks have degree strictly less than $\dim G + \dim C_G(u)$.\end{proof}




\end{document}